\documentclass[a4paper,10pt,notitlepage,reqno]{article}
\usepackage[english]{babel}
\usepackage[latin1]{inputenc}
\usepackage[T1]{fontenc}
\usepackage{amssymb,amsthm,amsmath} 
\usepackage{mathtools}
\usepackage{mathrsfs}
\usepackage{enumitem}
\usepackage{authblk}
\usepackage{color}
\usepackage{geometry}
\usepackage{hyperref}
\newcommand{\sii}{L^2}

\newcommand{\der}{\mathrm{d}}
\newcommand{\Hilbert}{\mathcal{H}}

\begingroup
    \makeatletter
    \@for\theoremstyle:=definition,remark,plain\do{%
        \expandafter\g@addto@macro\csname th@\theoremstyle\endcsname{%
            \addtolength\thm@preskip\parskip
            }%
        }
\endgroup
\newtheorem{Theorem}{Theorem}

\theoremstyle{definition}

\def\OMIT#1{}
\usepackage[normalem]{ulem}
\definecolor{DarkGreen}{rgb}{0,0.5,0.1} 

\newcommand\soutD{\bgroup\markoverwith
{\textcolor{DarkGreen}{\rule[.5ex]{2pt}{1pt}}}\ULon}
\newcommand\soutP{\bgroup\markoverwith
{\textcolor{blue}{\rule[.5ex]{2pt}{1pt}}}\ULon}
\newcommand{\Hm}[1]{\leavevmode{\marginpar{\tiny%
$\hbox to 0mm{\hspace*{-0.5mm}$\leftarrow$\hss}%
\vcenter{\vrule depth 0.1mm height 0.1mm width \the\marginparwidth}%
\hbox to
0mm{\hss$\rightarrow$\hspace*{-0.5mm}}$\\\relax\raggedright #1}}}

\newcommand{\Com}{\mathbb{C}}
\newcommand{\dom}{\mathop{\mathsf{dom}}}
\begin{document}
%

\title{\textbf{\LARGE
Semi-Dirac semi-metals quantum dots
}}
\author{Tuyen Vu }

\affil{Faculty of Mathematics and Informatics, Hanoi University of Science and Technology, Hanoi, Vietnam; tuyen.vuthibich@hust.edu.vn}

\date{\small 
30 July 2026
}
\maketitle
\vspace{-5ex} 
\begin{abstract}
\noindent
Novel findings on nanostructures in semi-metal  and semi-graphene materials are discussed regarding  rectangular quantum dots with  zigzag edges, MIT bag models, and Dirichlet boundary conditions.
The article theoretically investigates some models of quantum dots for these hybrid materials and the point spectra of the semi-Dirac semi-Laplacians  via  one-dimensional Laplace and Dirac operators. We also prove the self-adjointness of the semi-Laplacian defined on a rectangle, subject to zigzag and hard-wall boundary conditions, and study its spectral properties.
\end{abstract}


%

\section{Motivation}
%
Semi-Dirac semi-metals  are novel artificial materials that have abstracted a lot of research interests,  (see \cite{Banerjee_2009,Delplace-Montambaux_2010,Pardo-Pickett_2009,Pardo-Pickett_2010,Xiong,Zhou} and references therein). The quasi-particles nanostructures exhibit a hybrid combination of conventional zero-gap semiconductors in one  direction and  artificial materials like graphene in the orthogonal one. The novel materials disperse  quadratically and linearly  by virtue of the initiated anisotropic band structures.

Motivate from the mathematical  work \cite{KA}, where the authors  establish a tight-binding model, regardless of all the physical constants  via the Hamiltonian 
\begin{equation}\label{1}
  D := 
  \begin{pmatrix}
    -i\partial_y & -\partial_x^2 + \delta \\
    -\partial_x^2 + \delta & i\partial_y
  \end{pmatrix} \,,
\end{equation}
$D$ acts on
the Hilbert space $\Hilbert := \sii(\Omega)^2$
consisting of all $\Com^2$-valued spinors
$$
  \psi = 
  \begin{pmatrix}
    \psi_1 \\ \psi_2
  \end{pmatrix}
  \qquad \mbox{such that} \qquad
  \|\psi\|_{\Hilbert}^2 := \int_{\Omega} |\psi|^2 < \infty 
  \,,
$$
here $\delta$ is a constant which stands for the gap parameter, we use the  Euclidean norm $|\psi|:=\sqrt{|\psi_1|^2+|\psi_2|^2}$
and  denote by $\sii(\Omega)$ the  space
of Lebesgue square-integrable functions defined on the domain~$\Omega$. In the present article, we study  spectral properties of the semi-Dirac semi-Laplacian $D$ on rectangular domains. Moreover, we consider the Hamiltonian $G$  in semi-graphene semi-metals, which is analogously introduced in \cite{Banerjee_2009,Delplace-Montambaux_2010}.
$G$ can be considered  as the appropriate low energy description of quasiparticles dispersing linearly in one direction and quadratically in the perpendicular one, similarly to $D$. 
\begin{equation}\label{2}
  G := 
  \begin{pmatrix}
   0 &  -\partial_y - \partial_x^2 + \delta   \\
  \partial_y - \partial_x^2 + \delta  & 0
  \end{pmatrix} \,,
\end{equation}
This description seems more relevant to the mathematical behavior of graphene-metals particles, where two variables are rigorously combined as a sum of linear and quadratic elements in \eqref{1}, and also guarantees the operator self-adjointness  in $\mathbb{R}^2$  as  shown  in \cite{KA} by means of the Fourier transforms. Moreover, the self-adjointness is simultaneously valid for rectangular domains, subject to the mixed boundary conditions, i.e., MIT bag model and Dirichlet boundary conditions. The MIT bag model is a crucial model suggested in 1970s to study confinement of quarks in dimension three and it can be interpreted as infinite mass boundary conditions, see \cite{Arri,J.B,Cho1,Cho2}. In two dimensions, Dirac operators with MIT bag model are described by the boundary conditions $\psi_2 = i (n_1 + i n_2) \psi_1$ in the trace sense where $n= 
  \begin{psmallmatrix}
  n_1 \\ n_2 
  \end{psmallmatrix}
  $ 
stands for the outward unit normal of the domain, and other special boundary conditions such as zigzag boundary conditions are used in the description of graphene, cf. \cite{Ak,Ben1,Ben2,Hanna}.
Our work is  motivated by the work  \cite{Sch}, which constructs local boundary conditions to realise a Dirac differential expression as a self-adjoint operator on a regular bounded domain.

Recall that if we consider  $\Omega$ is a  bounded domain in $\mathbb{R}^2$ then the Dirichlet boundary conditions arise in a quantum system where non-relative particles constrained to a semiconductor nanostructure with hard-wall boundaries. On the other hand, if the lattice termination border of a graphene quantum dot is perpendicular to one of the neighbor bonds then zigzag boundary conditions naturally occur. Besides, MIT bag models or infinite mass boundary conditions expose as an idealisation of the system as the limit of the Dirac operators  with a large finite mass-term localised outside a quantum dot. We will investigate
 some novel semi-metals models as those on a cylinder or cone for nanoribbons in \cite{FS} where the self-adjointness and the point spectrum of the graphene operator on a rectangle with zigzag and periodic conditions are completely shown. In the article, Section \ref{2} introduces the semi-Dirac semi-Laplacian defined on rectangular dots with zigzag edges and hard-wall boundaries, investigates the eigenvalues of the operator $D$, proves the self-adjointness of $G$ defined on  rectangles and verifies that $0$ lies in the essential spectrum of $G$, and thus we get that the spectrum of $G$ is not purely discrete. In Section \ref{3}, we describe the semi-Dirac semi-metals defined on rectangles, subject to infinite mass or MIT bag model and Dirichlet boundary conditions and investigate the point spectrum of the semi-graphene operator.

\section{Rectangular dots with zigzag edges and hard-wall boundaries}\label{2}
Let us take $\Omega$ as
 a rectangular quantum dot enclosed by  the lines $x=0, x=a$ and zigzag-edge terminations at $y=0$ and $y=b$, provided that $a, b > 0$. Denote by $\gamma_1, \gamma_2, \gamma_3, \gamma_4$ the corresponding partial boundaries of $\Omega$ with respect to the zigzag edges at $y=0, y=b$ and the line segments at $x= 0, x=a$. We introduce the semi-Dirac problem, subject to the zigzag and Dirichlet boundary conditions 
\begin{equation}
\begin{aligned}
 D &:= 
  \begin{pmatrix}
    -i\partial_y & -\partial_x^2 + \delta \\
    -\partial_x^2 + \delta & i\partial_y
  \end{pmatrix} \,,\\
\dom D := \{
    \psi = 
  \begin{pmatrix}
    \psi_1 \\ \psi_2
  \end{pmatrix} \in \Hilbert : \
    \psi_1 \in H^1(\Omega; \mathbb{C}),  \, &\psi_2 \in H^1_0(\Omega; \mathbb{C}),\quad \partial_x^2\psi \in \Hilbert, \psi_1 = 0 \in H^\frac{1}{2}(\gamma_3), \psi_1 =0 \in H^\frac{1}{2}(\gamma_4)
  \} \,,
\end{aligned}
\end{equation}
where $H^1(\Omega; \mathbb{C}), H^1_0(\Omega; \mathbb{C})$ stand for the complex-valued $L^2$-based Sobolev space of one time weakly differentiable functions and the closure of  the space of infinitely differentiable functions with compact support on $\Omega$ or the test functions  $C^\infty_0(\Omega; \mathbb{C})$ in $H^1(\Omega; \mathbb{C})$, respectively. We also define $H^1(\Omega; \mathbb{C}^2)$ as a Sobolev space consisting of $\mathbb{C}^2$-valued functions with square-integrable first derivatives.
It can be seen as usual that the operator is not self-adjoint in $H^1$-setting for zigzag boundary conditions, see \cite{Hanna,FS}. 

With the aim of demonstrating the non-self-adjoint of the operator, we prove that the symmetry of this operator is invalid. Indeed, let us take a spinor $u = 
  \begin{pmatrix}
    u_1 \\ 0
  \end{pmatrix} \in \dom D$ such that $u_1 =0$ on $\gamma_1$, $u_1 > 0$ a.e. on $\gamma_2$.
Evaluating the following integral by integration by parts, we have
\begin{equation}
\begin{aligned}
(u, Du)&= \int_\Omega \overline{u_1} (-i \partial_y u_1) \der x \, \der y \\
&= \int_\Omega i \partial_y \overline{u_1} u_1 \der x \, \der y - \int_0^a i \overline{u_1} u_1|_{\gamma_2} + i \int_0^a \overline{u_1}u_1|_{\gamma_1} \\
&= (Du, u) - i\|u_1\|^2_{L^2(\gamma_2)}\\
&\neq (Du, u) \,.
\end{aligned}
\end{equation}
The difference appears due to the fact that the $L^2$-norm of $u_1$ on the boundary $\gamma_2$ is rigorously positive based on the establishment of this function. As a result, the operator is not symmetric. It follows that $D$ is not  self-adjoint, see \cite{Reed}. As the century turned, physicists seek to extend quantum mechanics by considering observables represented by non-self-adjoint operators and many fascinating observations come from physics of non-self-adjoint systems. Thus, the subject has become increasingly intriguing in recent years. 

If we now consider $\Omega$ as  the usual rectangle enclosed by the lines $x=0, x=a, y=0, y=b$ as analogous examples shown in \cite{FS}.   Using the similar approach, we derive  that $D$ is not self-adjoint. 
%
\begin{Theorem}
$\alpha_{n,k} = \pm \sqrt{\big(\delta + \frac{n^2\pi^2}{a^2}\big)^2 + \frac{k^2\pi^2}{b^2}}$ are the eigenvalues of $D$, where $ 1\leq n \in \mathbb{N}, k \in \mathbb{N}$.
\end{Theorem}

\begin{proof}
Let us consider the spectrum problem of the one-dimensional self-adjoint Laplacian defined on $[0,a]$, subject to Dirichlet boundary conditions
\begin{equation}
 (-\partial^2_x + \delta)\varphi(x) = \beta \varphi(x) \,,
\end{equation}
then the spectrum  is purely discrete and explicitly computed, see \cite{Henrot,Laugesen}
$$\delta < \beta_1 \leq \beta_2 \leq \beta_3 \leq ...\longrightarrow +\infty \,.$$
Here all $\beta_i, i\geq 1$ are the isolated eigenvalues of this operator. Especially, the squares of resonant frequencies of vibrating string with fixed edges are given by $\big\{\beta_n = \delta + \frac{n^2\pi^2}{a^2}\big\}_{n = 1}^\infty$ and the corresponding states $ \{\varphi_n(x)\}_{n \geq 1}$ of the quantum system are explicitly computed as follows
\begin{equation}
\varphi_n(x) = \sqrt{\frac{2}{a}} \sin \frac{n \pi x}{a} \,.
\end{equation}

Moreover, we investigate the spectrum of the one-dimensional Dirac operator employed on $[0,b]$ with zigzag boundary conditions
\begin{equation}\label{1D Dirac zigzag}
 \left\{
 \begin{aligned}
    -i\partial_y \psi_1  + \beta \psi_2 &= \lambda \psi_1 \,,\\
\beta\psi_1 +    i\partial_y \psi_2 &= \lambda \psi_2 \,,\\
    \psi_1 \in  H^1\big((0,b); \mathbb{C}\big) &\,,     \psi_2 \in  H^1_0\big((0,b); \mathbb{C}\big) 
   \,.
    \end{aligned}
  \right.
\end{equation}
We notice that this Dirac operator, subject to zigzag boundary conditions is not self-adjoint   in $H^1$-setting. Here $\lambda$ is an eigenvalue of the Dirichlet Laplacian 
$$ -\partial^2\psi_2 + \beta^2 \psi_2 = \lambda^2 \psi_2 \,.$$ 
As a result, we have the solution $$\psi_2(y) = C_1 \cos\sqrt{\lambda^2-\beta^2}y + C_2\sin\sqrt{\lambda^2-\beta^2}y \,,$$
where $\psi_2(0)= \psi_2(b) = 0$. It follows that $$C_1 = 0, \lambda = \pm \sqrt{\beta^2 + \frac{k^2\pi^2}{b^2}}, k \in \mathbb{N} \,.$$

Analogously, we also have
$$\psi_1 = C_3 \cos\sqrt{\lambda^2-\beta^2}y + C_4 \sin\sqrt{\lambda^2-\beta^2}y \,.$$
Substituting these identities back into System \eqref{1D Dirac zigzag}, we obtain
$$ C_2 = \frac{\beta C_4}{\lambda} \,, C_3 = -i \sqrt{\lambda^2-\beta^2} \frac{C_4}{\lambda} \,.$$
Hence, $\lambda = \pm \sqrt{\beta^2 + \frac{k^2\pi^2}{b^2}}, \beta = \delta + \frac{n^2\pi^2}{a^2}, 1\leq n \in \mathbb{N}, k \in \mathbb{N}$, and thus 
$\alpha_{n,k} = \pm \sqrt{\big(\delta + \frac{n^2\pi^2}{a^2}\big)^2 + \frac{k^2\pi^2}{b^2}}$
 are the eigenvalues of $D$.
\end{proof}


Denote by $\mathcal{X}$ the closure of $C^\infty_0(\Omega)$ the space of smooth functions with compact support in $\Omega$, which is endowed with the norm  $\|.\|_\mathcal{X} = \sqrt{\|.\|^2_{H^1(\Omega)} + \|\partial^2_x.\|^2_{L^2(\Omega)}}$.
Now we introduce the self-adjoint operator, which is the observable in the semi-graphene semi-conductor system 
\begin{equation}
\begin{aligned}
  \quad \quad & \quad \quad \quad \quad  G := 
   \begin{pmatrix}
   0 &  -\partial_y - \partial_x^2 + \delta   \\
  \partial_y - \partial_x^2 + \delta  & 0
  \end{pmatrix} \,,\\
 & \dom G := \{
    \psi = 
  \begin{pmatrix}
    \psi_1 \\ \psi_2
  \end{pmatrix} \in \Hilbert : \
    G\psi \in \Hilbert,
\psi_2 \in \mathcal{X}     
  \} \,.
  \end{aligned}
\end{equation}
\begin{Theorem}
$(G,\dom G)$ is a self-adjoint operator on a rectangular dot.
\end{Theorem}
\begin{proof}
It is apparent that the operator is symmetric, and thus $G \subset G^*$, where $G^*$ is the adjoint operator of $G$.

If $u=\begin{pmatrix}
u_1\\
u_2
\end{pmatrix}
\in \dom G^*$ then for all $v=\begin{pmatrix}
v_1\\v_2
\end{pmatrix}
\in \dom G$, there exists $\eta \in L^2\big(\Omega; \mathbb{C}^2\big)$ such that
\begin{equation}\label{adjoint identity}
(u, G v)=(\eta, v) \,.
\end{equation}
For all $v \in C^\infty_0(\Omega) \subset \dom G$, we derive that $G u \in \Hilbert$ due to \eqref{adjoint identity}. 
Now if we consider for all $v=\begin{pmatrix}
v_1\\0
\end{pmatrix}
\in \dom G$ and denote the trivial extension of $u_2$ to $\mathbb{R}^2$ be $\tilde{u}_2$ and $ A:= \partial_y - \partial_x^2 + \delta$ then for all $\varphi = \begin{pmatrix} \varphi_1 \\0 \end{pmatrix} \in C^\infty_0\big(\mathbb{R}^2; \mathbb{C}^2\big)$
\begin{align}
\left\langle A\tilde{u}_2, \varphi_1\right\rangle &= \langle \tilde{u}_2, A\varphi_1 \rangle
= \int_\Omega \overline{u_2} A\varphi_1 \,,\\
&= (u, G\varphi) = (\eta, \varphi) \leq C \|\varphi_1\|_{L^2(\mathbb{R}^2)}\,,
\end{align}
where $\langle\cdot,\cdot\rangle$ stands for the duality bracket of distributions. We deduce that $A \tilde{u}_2 \in L^2(\mathbb{R}^2; \mathbb{C})$. Employing the Fourier transform, it yields that $\tilde{u}_2 \in H^1(\mathbb{R}^2; \mathbb{C}), \partial_x \tilde{u}_2 \in L^2(\mathbb{R}^2; \mathbb{C})$. By virtue of the definition of $\tilde{u}_2$, we obtain that $u_2 \in \mathcal{X}$. It concludes the proof for the self-adjointness of $G$.
\end{proof}

Now we are in a position to investigate the spectrum of $(G,\dom G)$. If $(\lambda,\psi)$ is an eigenpair of the operator then we have the system
\begin{equation}\label{eigenproblem1}
\left\{
\begin{aligned}
(-\partial_y -\partial^2_x + \delta) \psi_2 &= \lambda \psi_1 \,,\\
(-\partial_x^2 + \delta + \partial_y)\psi_1 &= \lambda \psi_2 \,,\\
\psi\in \Hilbert, \psi_2 \in \mathcal{X} \,.
\end{aligned}
\right.
\end{equation}
As usual for zigzag boundary conditions, we observe that $0$ is an eigenvalue embedded in the essential spectrum of $(G, \dom G)$.
\begin{Theorem}
$0$ is an embedded eigenvalue  of $(G, \dom G)$ and the spectrum of $(G, \dom G)$ is not purely discrete.
\end{Theorem}
\begin{proof}
If we choose $\lambda=\psi_2=0$ in \eqref{eigenproblem1} then we have
\begin{equation}\label{eigenproblem2}
\left\{
\begin{aligned}
(\partial_y  -\partial^2_x + \delta) \psi_1 &= 0 \,,\\
\psi_1\in L^2(\Omega) \,.
\end{aligned}
\right.
\end{equation}
Let $\eta$ be an any positive eigenvalue of the one-dimensional Laplacian $-\partial_x^2 v(x) = \eta v(x)$ defined on $[0,a]$ without any boundary conditions then the corresponding eigenvalue reads $v(x)= E_1\cos\sqrt{\eta}x + E_2\sin\sqrt{\eta}x$ with constants $E_1, E_2 \in \mathbb{C}$. It yields that $\psi_1= e^{-(\eta+\delta)y}v(x)$ is a solution  of system \eqref{eigenproblem2}. As a result, $0$ is an eigenvalue of the operator  $(G, \dom G)$ with infinite multiplicity. Hence, $0$ belongs to the essential spectrum of the operator, and thus the spectrum of $(G, \dom G)$ is not purely discrete.
\end{proof}

Therefore, we succeed in setting up another model that guarantees the self-adjointness of the Dirac observable $G$ on a rectangular quantum dot and show the spectral property of the operator $(G, \dom G)$.

\section{Rectangular dots with MIT bag models and Dirichlet boundary conditions}\label{3}
Now we establish a quantum model defined on a rectangle, which has  vertical hard-wall boundaries and MIT boundary conditions on the horizontal sides of the  rectangle. 
Let us define the semi-Dirac semi-metals, subject to a mixed form of Dirichlet and infinite mass boundary conditions, i.e.
\begin{equation}
  H := 
  \begin{pmatrix}
    -i\partial_y & -\partial_x^2 + \delta \\
    -\partial_x^2 + \delta & i\partial_y 
  \end{pmatrix} \,,
\end{equation}
with the domain 
\begin{multline*}
  \dom H := \{
    \psi = 
  \begin{pmatrix}
    \psi_1 \\ \psi_2
  \end{pmatrix}   \in H^1(\Omega; \mathbb{C}^2),  \,  \quad \partial_x^2\psi \in \Hilbert, \psi_1 = \psi_2 =0 \in H^\frac{1}{2}(\gamma_3),  \psi_1 = \psi_2 =0 \in H^\frac{1}{2}(\gamma_4),\\
    \psi_1 = \psi_2  \in H^\frac{1}{2}(\gamma_1),     \psi_1 = -\psi_2  \in H^\frac{1}{2}(\gamma_2)
  \} \,.
\end{multline*}

...

%
%



\providecommand{\bysame}{\leavevmode\hbox to3em{\hrulefill}\thinspace}
\providecommand{\MR}{\relax\ifhmode\unskip\space\fi MR }
\providecommand{\MRhref}[2]{%
  \href{http://www.ams.org/mathscinet-getitem?mr=#1}{#2}
}
\providecommand{\href}[2]{#2}

\end{document}